\documentclass[pdflatex,sn-mathphys-num]{sn-jnl}

\usepackage{graphicx}%
\usepackage{multirow}%
\usepackage{amsmath,amssymb,amsfonts}%
\usepackage{amsthm}%
\usepackage{mathrsfs}%
\usepackage[title]{appendix}%
\usepackage{xcolor}%
\usepackage{textcomp}%
\usepackage{manyfoot}%
\usepackage{booktabs}%
\usepackage{algorithm}%
\usepackage{algorithmicx}%
\usepackage{algpseudocode}%
\usepackage{listings}%
\theoremstyle{thmstyleone}%
\newtheorem{thm}{Theorem}%
\newtheorem{prp}[thm]{Proposition}%
\newtheorem{crl}[thm]{Corollary}%
\theoremstyle{thmstyletwo}%
\newtheorem{rmk}{Remark}%
\theoremstyle{thmstylethree}%
\newtheorem{dfn}{Definition}%
\newtheorem{asm}[dfn]{Assumption}%

\usepackage{subcaption}
\usepackage{here}
\usepackage{physics}
\usepackage{mathtools}
\newcommand{\e}[1]{\begin{equation}#1\end{equation}}
\newcommand{\ald}[1]{\begin{aligned}#1\end{aligned}}
\newcommand{\ite}[1]{\begin{itemize}#1\end{itemize}}
\newcommand{\prn}[1]{\left(#1\right)}
\newcommand{\ang}[1]{\left\langle#1\right\rangle}
\newcommand{\cur}[1]{\left\{#1\right\}}
\def\coloneq{\mathrel{\mathop:}=}
\def\r{\mathbb{R}}
\def\rn{\mathbb{R}^n}
\allowdisplaybreaks

\begin{document}

\title[Pseudo-concave optimization of the first eigenvalue]{Pseudo-concave optimization of the first eigenvalue of elliptic operators with application to topology optimization by homogenization}

\author*[1]{\fnm{Akatsuki} \sur{Nishioka}}\email{nishioka.a.2122@m.isct.ac.jp}

\affil*[1]{\orgdiv{Department of Mathematical and Computing Science}, \orgname{School of Computing, Institute of Science Tokyo}, \orgaddress{\street{Ookayama 2-12-1}, \city{Meguro-ku}, \postcode{152-8550}, \state{Tokyo}, \country{Japan}}}

\abstract{
We study optimization problems for the first eigenvalue of a linear elliptic operator. As applications, we consider homogenized two-phase optimal design problems, also known as topology optimization problems, for conductivity and simplified elasticity settings. Under suitable assumptions, we prove that the first eigenvalue is pseudo-concave with respect to the density-like parameter. This pseudo-concavity implies that every stationary point of the corresponding maximization problem is a global maximizer. Also, for a certain pseudo-concave minimization problem in the conductivity setting, a classical $0$-$1$ minimizer exists. Finally, we present simple numerical experiments illustrating the theoretical results.
}



\maketitle

\section{Introduction}

A two-phase optimal design problem seeks to allocate two materials in a bounded open design domain $\Omega$. If material 2 occupies $\omega$ and material 1 occupies $\Omega\backslash\omega$, then a classical, or $0$-$1$, design is represented by a characteristic function $\chi_\omega\in L^\infty(\Omega;\{0,1\})$. Such unrelaxed problems often fail to admit solutions \cite{allaire02,casado15,casado17,casado22,murat71}. A standard remedy is relaxation by the homogenization method, in which $\chi_\omega$ is replaced by a density function $\theta\in L^\infty(\Omega;[0,1])$, and, when necessary, by additional homogenized tensors \cite{allaire02,casado22book,kohn86,tartar75}. This relaxation ensures the existence of solutions and gives computational tractability. However, relaxed problems are still nonconvex in some cases, and relaxed designs may contain gray-scale (density between $0$ and $1$). In this paper, we study these issues for optimization problems of the first (smallest) eigenvalue.

Homogenization-based topology optimization has a long history, but it has recently regained practical interest because modern manufacturing technologies, including 3D printing, can realize increasingly fine microstructured designs \cite{allaire19,groen20}. This paper also aims to present some classical theoretical results in a unified form that can be useful for applications.

One of the main problems is the following homogenized two-phase conductivity problem \cite{cox96}:
\e{\ald{
\label{p_intro}
& \underset{\theta\in L^\infty(\Omega)}{\max} 
& & \lambda_1(\theta)
=\underset{u\in H_0^1(\Omega)\backslash\{0\}}{\inf}
\frac{\int_{\mathrm{\Omega}} (c_1+(c_2-c_1)\theta(x))|\nabla u(x)|^2 dx}
{\int_{\mathrm{\Omega}}(\rho_1+(\rho_2-\rho_1)\theta(x))u(x)^2 dx} \\
& \mathrm{s.t.}
& & 0\le \theta(x)\le 1\ \ \mathrm{a.e.},\\
&
& & \int_{\mathrm{\Omega}} \theta(x) dx=\gamma,
}}
where $0<\gamma<|\Omega|$, $0<c_1<c_2$, and $0<\rho_1<\rho_2$. See Section \ref{subsec_cond_setting} for details.

Optimal design problems for the first eigenvalue have been widely studied in both mathematics and engineering \cite{allaire01,allaire02,bendsoe04,casado15,casado17,casado22,cox90,cox95,henrot06,jha11,laurain14,matsue15,mazari22}. Many existing works focus on existence, relaxation, and regularity of optimal designs. The focus of this paper is different: we use generalized convexity, more precisely pseudo-concavity, to identify classes of relaxed eigenvalue optimization problems in which every stationary point of a maximization problem is globally optimal.

The main theoretical result is that, under suitable concavity and convexity assumptions on the operators, the first eigenvalue is pseudo-concave with respect to the density-like parameter. The proof is based on the Clarke subdifferential and therefore covers the case where the first eigenvalue is multiple and not differentiable. This extends the differentiable result of \cite{jouron78} and the finite-dimensional symmetric-matrix result of \cite{nishioka23coap}.

We apply this result to homogenized optimal design problems of conductivity and elasticity. In the conductivity setting, we prove that, for several relaxed maximization problems, every stationary point is a global maximizer. We also mention that, for a certain pseudo-concave minimization problem, the existence of a classical $0$-$1$ minimizer is known by a problem-specific argument. We also conduct simple numerical experiments to illustrate these theoretical results and provide the FreeFEM code in Appendix \ref{apn}.

The fully relaxed linear elasticity problem remains more difficult. In that case, the homogenized tensor is generally nonlinear with respect to the density, and the pseudo-concavity result of this paper does not directly apply. The existence of suboptimal local optima in linear elastic optimal design, discussed for example in \cite{allaire02}, is an important theoretical issue and is also relevant in engineering applications, where eigenfrequency maximization is used to design structures with high dynamic stiffness. We regard the extension of the present approach to the genuine linear elasticity problem as a main direction for future work.

This paper is organized as follows. Section \ref{sec_theory} introduces pseudo-concavity and proves the pseudo-concavity of the first eigenvalue in an abstract generalized-eigenvalue setting. Section \ref{sec_app} applies the result to homogenized optimal design problems of conductivity and elasticity. Section \ref{sec_num} presents numerical experiments in the conductivity setting.

\subsection*{Notation}

We use the following notation.
\ite{
\item $|x|$ and $x\cdot y$ denote the Euclidean norm and the Euclidean inner product of vectors $x,y\in\rn$.
\item $A:B\coloneq\tr(A^\top B)$ and $|A|\coloneq\sqrt{\tr(A^\top A)}$ denote the inner product and the norm of matrices $A,B\in\r^{n\times n}$.
\item $\|u\|_V$ and $\ang{u^*,u}_{V^*,V}$ denote the norm and the dual coupling in a function space $V$, where $u\in V$, $u^*\in V^*$, and $V^*$ is the topological dual space of $V$. For a Hilbert space $H$, $\ang{u,v}_H$ denotes the inner product. We often omit the subscript.
\item $\mathrm{conv}\,S$ is the convex hull of a set $S\subseteq X$ where $X$ is a linear space.
\item $\otimes$ denotes the tensor product.
\item $L^\infty(\Omega;\{0,1\})$ and $L^\infty(\Omega;[0,1])$ are the spaces of functions in $L^\infty(\Omega)$ that take values in $\{0,1\}$ and $[0,1]$ almost everywhere, respectively.
\item Two different uses of $\partial$ should not be confused. $\partial\Omega$ is the boundary of a set $\Omega$ and $\partial f(\theta)$ is the Clarke subdifferential of a function $f$ at $\theta$.
\item We often omit the argument $(x)$ of functions such as $\theta$, $\chi_\omega$, and $u$.
}

\section{Pseudo-concavity of the first eigenvalue}\label{sec_theory}

\subsection{Preliminary results on pseudo-concave functions}

We introduce a possibly nonsmooth pseudo-concave function, which uses the Clarke subdifferential in its definition. The Clarke subdifferential is an extension of the ordinary convex subdifferential to a nonconvex function. It is defined through the generalized directional derivative.

\begin{dfn}[generalized directional derivative]
Let $X$ be a Banach space, $D\subseteq X$ be an open set, and $f:D\to\r$ be a locally Lipschitz continuous function. The generalized directional derivative of $f$ at $\theta\in D$ in the direction $d\in X$ is defined by 
\e{
f^\circ(\theta;d)\coloneq\underset{\theta'\to\theta,\ t\downarrow0}{\mathrm{lim\,sup}}\,\frac{f(\theta'+td)-f(\theta')}{t}.
}
\end{dfn}

\begin{dfn}[Clarke subdifferential]
Let $X$ be a Banach space, $D\subseteq X$ be an open set, and $f:D\to\r$ be a locally Lipschitz continuous function. The Clarke subdifferential of $f$ at $\theta\in D$ is the subset of $X^*$ defined by
\e{
\partial f(\theta)\coloneq\cur{g\in X^*\mid f^\circ(\theta;d)\ge\ang{g,d}_{X^*,X},\ \forall d\in X}.
}
Each element of $\partial f(\theta)$ is called a Clarke subgradient.
\end{dfn}

\begin{rmk}
When we consider a concave function $f$ instead of a convex function, we often use the superdifferential (instead of subdifferential) of $f$ at $\theta$ defined by $\partial^+ f(\theta)\coloneq-\partial (-f)(\theta)$ where $\partial (-f)(\theta)$ is the ordinary convex subdifferential for the convex function $-f$. In the case of the Clarke subdifferential, $\partial (sf)(\theta)=s\partial f(\theta)$ holds for any scalar $s$ ($s$ can be negative unlike ordinary convex subdifferentials) \cite[Proposition 2.3.1]{clarke90}, and thus there is no distinction between Clarke subdifferentials and Clarke superdifferentials.
\end{rmk}

We define a pseudo-concave function. An example of a pseudo-concave function that is not concave is a bell curve $f(x)=\exp(-x^2)$.

\begin{dfn}[pseudo-concave function \cite{penot97}]
Let $X$ be a Banach space and $D\subseteq X$ be an open convex set. A locally Lipschitz continuous function $f:D\to\r$ is said to be pseudo-concave in $D$ if, for any $\theta,\theta'\in D$ and for any Clarke subgradient $g\in\partial f(\theta)$, the following implication holds:
\e{
f(\theta)<f(\theta')\ \Rightarrow\ \ang{g,\theta'-\theta}>0.
\label{pseudo}
}
For a closed convex set $S$, we often say that $f$ is pseudo-concave in $S$ or $f:S\to\r$ is pseudo-concave if $f$ admits a locally Lipschitz extension $\widetilde f:D\to\r$ to some open convex set $D\supseteq S$ such that $\widetilde f$ is pseudo-concave in $D$. A function $f$ is said to be pseudo-convex if $-f$ is pseudo-concave.
\end{dfn}

Every concave function is a pseudo-concave function since the inequality of concavity $f(\theta')-f(\theta)\le\langle g,\theta'-\theta\rangle$ ensures the implication \eqref{pseudo}. Every pseudo-concave function is a quasi-concave function.\footnote{A function $f:X\to\r$ is called a quasi-concave function if its superlevel set $\cur{\theta\in X\mid f(\theta)\ge \alpha}$ is convex for any $\alpha\in\r$. This should not be confused with quasi-convexity in the sense of Morrey \cite{morrey52} used to ensure the weak lower-semicontinuity of an integral functional in the calculus of variations.}

We introduce Clarke stationarity, a necessary condition for local optimality \cite[Proposition 2.3.2]{clarke90}. In a pseudo-concave maximization problem (consisting of a pseudo-concave objective function and a closed convex feasible set), every Clarke stationary point is a global maximizer. This is a well-known fact (e.g., \cite{penot97}) and is directly derived from the definitions of a pseudo-convex function and a Clarke stationary point, but we give a proof for self-containment.

\begin{dfn}[Clarke stationary point]
Let $X$ be a Banach space, $D\subseteq X$ be an open set, $S\subset D$ be a nonempty closed convex set, and $f:D\to\r$ be a locally Lipschitz continuous function. A Clarke stationary point of a maximization problem
\e{
\underset{\theta\in S}{\max}\ f(\theta)
}
is a point $\theta^*\in S$ satisfying
\e{
\ang{g,\theta-\theta^*}\le 0
\label{stationarity}
}
for some $g\in\partial f(\theta^*)$ and for any $\theta\in S$.
\end{dfn}

\begin{prp}[property of solutions to pseudo-concave maximization]\label{prp_psemax}
Let $X$ be a Banach space, $D\subseteq X$ be an open convex set, $S\subset D$ be a nonempty closed convex set, and $f:D\to\r$ be a pseudo-concave function. Every Clarke stationary point of a pseudo-concave maximization problem
\e{
\underset{\theta\in S}{\max}\ f(\theta)
}
is a global maximizer.
\end{prp}
\begin{proof}
Let $\theta^*$ be a Clarke stationary point. Suppose there exists $\hat{\theta}\in S$ such that $f(\theta^*)<f(\hat{\theta})$. By substituting $\theta^*$ to $\theta$ and $\hat{\theta}$ to $\theta'$ in \eqref{pseudo}, we obtain $\langle g,\hat{\theta}-\theta^*\rangle>0$ for any $g\in\partial f(\theta^*)$, which contradicts the Clarke stationarity \eqref{stationarity} of $\theta^*$.
\end{proof}

\begin{rmk}\label{rmk_no_general_extreme}
In contrast to finite-dimensional compact settings, pseudo-concavity alone does not imply that a minimizer over a closed convex subset of an infinite-dimensional Banach space can be chosen as an extreme point. The existence of a classical $0$-$1$ minimizer, which belongs to the set of extreme points of the feasible set, in Proposition \ref{crl_den_rel} is proved by a problem-specific argument.
\end{rmk}

\subsection{Main theoretical results}

We show the pseudo-concavity of the first eigenvalue of linear elliptic operators in the abstract setting of \cite[Section 3]{cox95}. See also \cite[Section 8.6]{attouch14} for an abstract theory of eigenvalues of linear elliptic operators. We consider the following eigenvalue problem
\e{
A(\theta)u=\lambda B(\theta)u,
\label{eig_gen_0}
}
where the first eigenvalue can be written using the Rayleigh quotient
\e{
\lambda_1(\theta)=\underset{u\in V\backslash\{0\}}{\inf}\,\frac{\ang{A(\theta)u,u}_H}{\ang{B(\theta)u,u}_H}.
\label{eig_gen}
}
Throughout this paper, expressions such as $\ang{A(\theta)u,u}_H$ are understood in the variational sense through the corresponding symmetric bilinear forms of an elliptic differential operator. Problem \eqref{eig_gen_0} is often called a generalized eigenvalue problem since an operator appears in both the left- and right-hand side. Optimization problems of generalized eigenvalues of symmetric matrices are studied in \cite{achtziger07siam,nishioka23coap,nishioka24svva}.

Assumptions on the operators $A,B$ are as follows.

\begin{asm}\label{asm1}
Let $V,H$ be Hilbert spaces such that $V$ is densely and compactly embedded in $H$ (a typical example is $V=H^1_0(\Omega)$ and $H=L^2(\Omega)$). Let $X$ be a Banach space. Let $D\subseteq X$ be an open convex set and let $S\subseteq D$ be a nonempty convex set. A space of bounded linear self-adjoint operators from $V$ (resp.~$H$) to itself is denoted by $\mathscr{B}_s(V,V)$ (resp.~$\mathscr{B}_s(H,H)$). Let $A:D\to\mathscr{B}_s(V,V)$ and $B:D\to\mathscr{B}_s(H,H)$ be maps. Operators $A(\theta)$ and $B(\theta)$, $\theta\in D$, satisfy the following properties:
\ite{
\item Ellipticity and boundedness: $\underline{a} \|u\|_V^2\le \ang{A(\theta)u,u}_H\le \overline{a}\|u\|_V^2$ and\\
$\underline{b} \|u\|_H^2\le \ang{B(\theta)u,u}_H\le \overline{b}\|u\|_H^2$ for any $u\in V$ and $\theta\in D$, where $0<\underline{a}<\overline{a}$ and $0<\underline{b}<\overline{b}$ are constants.
\item Concavity and convexity: $\theta\mapsto\ang{A(\theta)u,u}_H$ is concave and $\theta\mapsto\ang{B(\theta)u,u}_H$ is convex in $D$ for any $u\in V$.
\item Smoothness: $\theta\mapsto A(\theta),B(\theta)$ are continuously differentiable in $D$.
}
\end{asm}

The eigenvalues form a positive nondecreasing sequence with finite multiplicities that tends to $+\infty$ as $n\to\infty$: $0<\lambda_1\le\lambda_2\le\ldots\le\lambda_n\ldots$ (see \cite[Section 3]{cox95} or \cite[Section 8.6]{attouch14}). The Clarke subdifferential of the first eigenvalue is as follows.

\begin{prp}[Clarke subdifferential of the first eigenvalue (cf.~{\cite[Equation (3.2)]{cox95}})]\label{prp_cla}
Under Assumption \ref{asm1}, the first eigenvalue defined by \eqref{eig_gen} is locally Lipschitz continuous. Define $\mathcal{S}_\theta\coloneq\cur{u\in\mathscr{E}_1(\theta)\mid \ang{B(\theta)u,u}_H=1}$, where $\mathscr{E}_1(\theta)\subseteq V$ is the eigenspace of $\lambda_1(\theta)$. For $u\in\mathcal{S}_\theta$, define $\Phi_{\theta,u}\in X^*$ by
\e{
\ang{\Phi_{\theta,u},d}_{X^*,X}
\coloneq
\ang{\prn{\mathrm{D}A(\theta)[d]-\lambda_1(\theta)\mathrm{D}B(\theta)[d]}u,u}_H,
\quad d\in X.
}
Then the Clarke subdifferential at $\theta\in D$ is
\e{
\partial\lambda_1(\theta)=\mathrm{conv}\cur{\Phi_{\theta,u}\mid u\in\mathcal{S}_\theta},
}
where $\mathrm{D}A(\theta):X\to \mathscr{B}_s(V,V)$ and $\mathrm{D}B(\theta):X\to \mathscr{B}_s(H,H)$ are the Fr\'{e}chet derivatives of $A,B$ at $\theta$.
\end{prp}

\begin{rmk}
In \cite[Equation (3.2)]{cox95}, the first eigenspace is intersected with the
$H$-unit sphere. However, if the $H$-unit sphere is used, the corresponding formula should contain the factor
$\ang{B(\theta)u,u}_H$ in the denominator. In Proposition \ref{prp_cla}, we instead
use the $B(\theta)$-normalized unit sphere in order to absorb this denominator and
simplify the notation.
\end{rmk}

We show the pseudo-concavity by using the Clarke subdifferential.

\begin{thm}[pseudo-concavity of the first eigenvalue]\label{thm_main}
Under Assumption \ref{asm1}, the first eigenvalue $\theta\mapsto\lambda_1(\theta)$ defined by \eqref{eig_gen} is pseudo-concave in $S$.
\end{thm}
\begin{proof}
Let $\theta,\theta'\in S$ satisfy
\e{
\lambda_1(\theta)<\lambda_1(\theta').
}
We prove that
\e{
\ang{g,\theta'-\theta}_{X^*,X}>0
}
for every $g\in\partial\lambda_1(\theta)$.

By Proposition \ref{prp_cla}, for each $g\in\partial\lambda_1(\theta)$, there exist $M\in\mathbb{N}$, coefficients $c_j\ge0$ satisfying $\sum_{j=1}^M c_j=1$, and functions $u_j\in\mathcal{S}_\theta$ such that
\e{
g=\sum_{j=1}^M c_j\Phi_{\theta,u_j}.
}
Fix $j\in\cur{1,\ldots,M}$. Since $\theta\mapsto\ang{A(\theta)u_j,u_j}_H$ is concave, we have
\e{
\ang{\mathrm{D}A(\theta)[\theta'-\theta]u_j,u_j}_H
\ge
\ang{A(\theta')u_j,u_j}_H-
\ang{A(\theta)u_j,u_j}_H.
}
Since $\theta\mapsto\ang{B(\theta)u_j,u_j}_H$ is convex and $\lambda_1(\theta)>0$, we also have
\e{
-\lambda_1(\theta)\ang{\mathrm{D}B(\theta)[\theta'-\theta]u_j,u_j}_H
\ge
-\lambda_1(\theta)
\prn{
\ang{B(\theta')u_j,u_j}_H-
\ang{B(\theta)u_j,u_j}_H
}.
}
Combining these inequalities gives
\e{\ald{
\ang{\Phi_{\theta,u_j},\theta'-\theta}_{X^*,X}
&\ge
\ang{A(\theta')u_j,u_j}_H
-\lambda_1(\theta)\ang{B(\theta')u_j,u_j}_H \\
&\quad-
\prn{
\ang{A(\theta)u_j,u_j}_H
-\lambda_1(\theta)\ang{B(\theta)u_j,u_j}_H
}\\
&=\ang{A(\theta')u_j,u_j}_H
-\lambda_1(\theta)\ang{B(\theta')u_j,u_j}_H,
}}
since $u_j\in\mathscr{E}_1(\theta)$, and we have
\e{
\ang{A(\theta)u_j,u_j}_H
=
\lambda_1(\theta)\ang{B(\theta)u_j,u_j}_H.
}

By the Rayleigh quotient characterization of $\lambda_1(\theta')$,
\e{
\lambda_1(\theta')
\le
\frac{
\ang{A(\theta')u_j,u_j}_H
}{
\ang{B(\theta')u_j,u_j}_H
}.
}
Therefore,
\e{\ald{
\ang{\Phi_{\theta,u_j},\theta'-\theta}_{X^*,X}
&\ge
\prn{\lambda_1(\theta')-\lambda_1(\theta)}
\ang{B(\theta')u_j,u_j}_H\\
&>0,
}}
where the last inequality follows from $\lambda_1(\theta')>\lambda_1(\theta)$, $u_j\ne0$, and the positive definiteness of $B(\theta')$.

Consequently,
\e{
\ang{g,\theta'-\theta}_{X^*,X}
=
\sum_{j=1}^M c_j
\ang{\Phi_{\theta,u_j},\theta'-\theta}_{X^*,X}
>0.
}
This proves the implication \eqref{pseudo}, and hence $\theta\mapsto\lambda_1(\theta)$ is pseudo-concave in $S$.
\end{proof}



\section{Applications: optimal design by homogenization}\label{sec_app}

\subsection{Basics of optimal design by homogenization}\label{sec_hom}

In this subsection, we briefly explain the relaxation of optimal design (topology optimization) problems by the homogenization method. See \cite{allaire02,allaire19,casado22book,kohn86} for the details.

We find an optimal design of two-phase material on a regular bounded domain $\Omega\subset\r^N$. The optimization variable (design variable) $\chi_\omega\in L^\infty(\Omega;\{0,1\})$ is a characteristic function of the subset $\omega\subseteq\Omega$ occupied by the material 2 with material constants (e.g.~$c_2,\rho_2$). Ideally, in topology optimization, material 1 occupying $\Omega\backslash\omega$ should be degenerate ($c_1=\rho_1=0$) and $\Omega\backslash\omega$ corresponds to holes. However, when material 1 is degenerate, the ellipticity in Assumption \ref{asm1} is not satisfied, and the treatment of the boundary condition of $\partial\omega$ becomes complicated (cf.~\cite{allaire02}). Thus, we consider weak material 1 (e.g.~$c_1,\rho_1>0$ with $c_1<c_2$ and $\rho_1<\rho_2$) instead of holes.

The admissible set of optimization variables of the original (unrelaxed) problem is 
\e{
\mathrm{ad}_\gamma\coloneq\cur{\chi_\omega\in L^\infty(\Omega)\ \middle|\ \chi_\omega=0\ \mathrm{or}\ 1\ \mathrm{a.e.},\ \int_{\mathrm{\Omega}} \chi_\omega dx=\gamma},
\label{ad_cla}
}
where $0<\gamma<|\Omega|$ is the designated volume of $\omega$. We say a solution (design) is classical or 0-1 if it belongs to $\mathrm{ad}_\gamma$. In the problem relaxed by the homogenization method, the admissible set \eqref{ad_cla} is replaced by its weak$^*$ closure in $L^\infty(\Omega)$, which is
\e{
\mathrm{ad}^*_\gamma\coloneq\cur{\theta\in L^\infty(\Omega)\ \middle|\ 0\le\theta\le1\ \mathrm{a.e.},\ \int_{\mathrm{\Omega}} \theta dx=\gamma}.
\label{ad_rel}
}
Namely, the optimization variable of the relaxed problem is the density function $\theta\in L^\infty(\Omega)$, which takes continuous values in $[0,1]$.

However, the relaxation by the homogenization method is not only replacing the characteristic function $\chi_\omega$ by the density function $\theta$, because the microstructure also affects the material properties. We explain the details with two different settings of optimal design problems: conductivity and elasticity. In the conductivity setting, the eigenvector $u$ is a scalar-valued function (often used to represent temperature or 1-dimensional displacement), and the first eigenvalue is simple if $\Omega$ is connected. In the elasticity setting, the eigenvector $u$ is a vector-valued function (often used to represent 2- or 3-dimensional displacement) and the first eigenvalue is not necessarily simple (not necessarily differentiable) \cite[Section 7.3.3]{allaire07}.

\subsection{Problem setting in conductivity}\label{subsec_cond_setting}

We assume that $\Omega\subset\r^N$ is a regular bounded open set. We consider the following eigenvalue problem of finding eigenvalues $\lambda\in\r$ and eigenvectors $u\in H_0^1(\Omega)$ satisfying (in the weak sense)
\e{
\begin{cases}
-\nabla\cdot c^+(\chi_\omega)\nabla u =\lambda\rho(\chi_\omega)u, & \text{in}\ \Omega\\
u=0 & \text{on}\ \partial \Omega
\end{cases}
\label{eigen}
}
where
\e{\ald{
& c^+(\chi_\omega)= c_1+(c_2-c_1)\chi_\omega,\\
& \rho(\chi_\omega)= \rho_1+(\rho_2-\rho_1)\chi_\omega,
\label{cond_dens}
}}
with the conductivity and density constants
\e{\ald{
& 0<c_1<c_2,\\
& 0<\rho_1<\rho_2.
}}
The equation \eqref{eigen} models heat conduction or vibration of a membrane \cite{casado22,cox90,cox96}. By the ellipticity, we have countably many eigenvalues satisfying $0<\lambda_1\le\lambda_2\le\ldots$ (see~\cite{attouch14,evans10,henrot06}). We aim to optimize the material layout $\chi_\omega$ to maximize the first eigenvalue satisfying \eqref{eigen}. By using the Rayleigh quotient, this problem can be written as
\e{\label{p_original}
\underset{\chi_\omega\in\mathrm{ad}_\gamma}{\sup}\ \ \underset{u\in H_0^1(\Omega)\backslash\{0\}}{\inf}\frac{\int_{\mathrm{\Omega}} c^+(\chi_\omega)|\nabla u|^2 dx}{\int_{\mathrm{\Omega}}\rho(\chi_\omega)u^2 dx},
}
which does not necessarily admit a solution \cite{allaire02,casado15,casado17,casado22,mazari22,murat71}.

We relax Problem \eqref{p_original} by the homogenization method to ensure the existence of solutions and make the maximum of the relaxed problem equal to the supremum of the original problem \eqref{p_original}. We replace the admissible set \eqref{ad_cla} by its weak$^*$ closure \eqref{ad_rel} in $L^\infty(\Omega)$. The integrand $\rho(\chi_\omega)u^2$ in the denominator is also replaced by $\rho(\theta)u^2$. The treatment of the integrand $c^+(\chi_\omega)|\nabla u|^2$ in the numerator is more complicated since it involves $\nabla u$, which is affected by not only $\theta$ but also the microstructure of the material distribution. It is replaced by its $H$-limit $A^*\nabla u\cdot\nabla u$ (for details of $H$-convergence, see \cite{allaire02,casado22book}), where $A^*$ is an $N\times N$ symmetric matrix called a homogenized tensor. In the end, the relaxed problem is
\e{\label{p_relaxed_pre}
\underset{\substack{\theta\in\mathrm{ad}^*_\gamma,\\A^*\in G_\theta}}{\max}\ \underset{u\in H_0^1(\Omega)\backslash\{0\}}{\inf}\frac{\int_{\mathrm{\Omega}} A^*\nabla u\cdot\nabla u dx}{\int_{\mathrm{\Omega}}\rho(\theta)u^2 dx},
}
where $G_\theta$ is the set of all homogenized tensors ($H$-limit) called the G-closure \cite[Section 2.2.3]{allaire02}.

We can simplify the relaxed problem \eqref{p_relaxed_pre} by eliminating the variable $A^*$. The simplest case of the Hashin--Shtrikman energy bound \cite[Section 2.2.2]{allaire02} gives
\e{
c^-(\theta)|\nabla u|^2 \le A^*\nabla u\cdot\nabla u \le c^+(\theta)|\nabla u|^2,\ \ \forall A^*\in G_\theta,
\label{HSbound}
}
where
\e{\ald{
& c^-(\theta)=\prn{\frac{1-\theta}{c_1}+\frac{\theta}{c_2}}^{-1}\\
& c^+(\theta)=c_1+(c_2-c_1)\theta
}}
are the harmonic and arithmetic means of $c_1$ and $c_2$. The equalities in \eqref{HSbound} are attained for some $A^*\in G_\theta$. According to \cite{lipton94}, we can interchange maximization w.r.t. $A^*$ and minimization w.r.t. $u$. It results in the pointwise maximization $\max_{A^*\in G_\theta} A^*\nabla u\cdot\nabla u$, which is equal to $c^+(\theta)|\nabla u|^2$ by \eqref{HSbound}. Therefore, the relaxed problem \eqref{p_relaxed_pre} is equivalent to
\e{\label{p_relaxed}
\underset{\theta\in\mathrm{ad}^*_\gamma}{\max}\ \ \underset{u\in H_0^1(\Omega)\backslash\{0\}}{\inf}\frac{\int_{\mathrm{\Omega}} c^+(\theta)|\nabla u|^2 dx}{\int_{\mathrm{\Omega}}\rho(\theta)u^2 dx},
}
which only depends on the variable $\theta$. Note that \eqref{p_relaxed} is the same problem as \eqref{p_intro}, but with simplified notation.

We obtain the following proposition, which justifies the relaxed problem \eqref{p_relaxed}.

\begin{prp}[{\cite[Proposition 2.3 and Corollary 4.2]{cox96}}]
Problem \eqref{p_relaxed} is the relaxed problem of the original problem \eqref{p_original} in the sense that there exists a solution to \eqref{p_relaxed} and its maximum coincides with the supremum of \eqref{p_original}.
\end{prp}

Next, instead of the maximization problem \eqref{p_original}, we consider the minimization problem 
\e{
\underset{\chi_\omega\in\mathrm{ad}_\gamma}{\min}\ \ \underset{u\in H_0^1(\Omega)\backslash\{0\}}{\inf}\frac{\int_{\mathrm{\Omega}} c^+(\chi_\omega)|\nabla u|^2 dx}{\int_{\mathrm{\Omega}}\rho(\chi_\omega)u^2 dx}.
}
The relaxed problem is
\e{
\underset{\substack{\theta\in\mathrm{ad}^*_\gamma,\\A^*\in G_\theta}}{\min}\ \underset{u\in H_0^1(\Omega)\backslash\{0\}}{\inf}\frac{\int_{\mathrm{\Omega}} A^*\nabla u\cdot\nabla u dx}{\int_{\mathrm{\Omega}}\rho(\theta)u^2 dx},
\label{p_relaxed_min0}
}
and it can also be simplified by interchanging minimization w.r.t.~$A^*$ and minimization w.r.t.~$u$. We obtain the pointwise minimization, which is equal to the harmonic mean by \eqref{HSbound}: $\min_{A^*\in G_\theta}\,A^*\nabla u\cdot\nabla u=c^-(\theta)|\nabla u|^2$. Thus, the relaxed problem \eqref{p_relaxed_min0} can be transformed into a simpler problem without $A^*$:
\e{\label{p_relaxed_min}
\underset{\theta\in\mathrm{ad}^*_\gamma}{\min}\ \underset{u\in H_0^1(\Omega)\backslash\{0\}}{\inf}\frac{\int_{\mathrm{\Omega}} c^-(\theta)|\nabla u|^2 dx}{\int_{\mathrm{\Omega}}\rho(\theta)u^2 dx}.
}
Note that the harmonic mean $c^-(\theta)$ instead of the arithmetic mean $c^+(\theta)$ appears in the numerator.

\subsection{Pseudo-concavity results in conductivity}\label{sec_cond}

We divide Problem \eqref{p_relaxed} into three cases: when the density $\theta$ appears both in the denominator and numerator of the Rayleigh quotient, when the density $\theta$ only appears in the denominator (when $c_1=c_2$), and when the density $\theta$ only appears in the numerator (when $\rho_1=\rho_2$). All three cases are studied in the literature \cite{casado15,casado17,casado22,cox96,henrot06,matsue15}.

\subsubsection{Density both in the denominator and numerator}

\begin{thm}\label{thm_psemax}
The relaxed problem \eqref{p_relaxed} is a pseudo-concave maximization problem, and thus every stationary point is a global maximizer.
\end{thm}
\begin{proof}
We take $X=L^\infty(\Omega)$, $S=\mathrm{ad}^*_\gamma$, and an open convex set $D\subset L^\infty(\Omega)$ containing $S$ so that the ellipticity condition in Assumption \ref{asm1} on $D$. Problem \eqref{p_relaxed} satisfies Assumption \ref{asm1} with $H=L^2(\Omega)$, $V=H^1_0(\Omega)$, $\ang{A(\theta)u,u}_H=\int_\Omega c^+(\theta)|\nabla u|^2dx$, and $\ang{B(\theta)u,u}_H=\int_\Omega \rho(\theta)u^2dx$. The ellipticity of the operator $A(\theta)$ follows from Poincar\'{e}'s inequality (cf.~\cite[Proposition 4.3.10]{allaire07}). Operators in the numerator and the denominator are both affine with respect to $\theta$. Therefore, it is a direct consequence of Theorem \ref{thm_main} and Proposition \ref{prp_psemax}.
\end{proof}

In contrast to the maximization problem \eqref{p_relaxed}, Theorem \ref{thm_main} does not apply to the minimization problem \eqref{p_relaxed_min}, since the numerator is convex, rather than concave, in $\theta$. This type of problem is studied in \cite{cox96}.

\subsubsection{Density only in the numerator}

When $\theta$ only appears in the numerator ($\rho_1=\rho_2=1$ without loss of generality), the relaxed maximization problem becomes
\e{
\underset{\theta\in\mathrm{ad}^*_\gamma}{\max}\ \underset{u\in H_0^1(\Omega)\backslash\{0\}}{\inf}\frac{\int_{\mathrm{\Omega}}c^+(\theta)|\nabla u|^2 dx}{\int_{\mathrm{\Omega}}u^2 dx},
\label{p_num_rel_max}
}
which is a concave maximization problem since the objective function is a minimum of concave functions. In this case, the non-existence of classical 0-1 solutions is shown in \cite{casado22}.

In contrast, the relaxed minimization problem
\e{
\underset{\theta\in\mathrm{ad}^*_\gamma}{\min}\ \underset{u\in H_0^1(\Omega)\backslash\{0\}}{\inf}\frac{\int_{\mathrm{\Omega}}c^-(\theta)|\nabla u|^2 dx}{\int_{\mathrm{\Omega}}u^2 dx},
\label{p_num_rel_min}
}
cannot be expected to be a pseudo-concave minimization problem since the numerator is convex, rather than concave, in $\theta$. This problem does not have classical 0-1 solutions, either \cite{casado15,casado17}.

\subsubsection{Density only in the denominator}

When $\theta$ only appears in the denominator ($c_1=c_2=1$ without loss of generality), the maximization problem
\e{
\underset{\theta\in\mathrm{ad}^*_\gamma}{\max}\ \underset{u\in H_0^1(\Omega)\backslash\{0\}}{\inf}\frac{\int_{\mathrm{\Omega}}|\nabla u|^2 dx}{\int_{\mathrm{\Omega}}\rho(\theta)u^2 dx}
\label{p_den_rel_max}
}
is also a pseudo-concave maximization problem and every stationary point is a global maximizer due to Theorem \ref{thm_psemax}.

In this case, the minimization problem
\e{
\underset{\theta\in\mathrm{ad}^*_\gamma}{\min}\ \underset{u\in H_0^1(\Omega)\backslash\{0\}}{\inf}\frac{\int_{\mathrm{\Omega}}|\nabla u|^2 dx}{\int_{\mathrm{\Omega}}\rho(\theta)u^2 dx},
\label{p_den_rel_min}
}
becomes a pseudo-concave minimization problem. Although a general extreme-point principle for pseudo-concave minimization is not available in this infinite-dimensional setting (see Remark \ref{rmk_no_general_extreme}), the existence of a classical minimizer, which is an extreme point\footnote{The fact that the set of all extreme points of $\mathrm{ad}^*_\gamma$ is $\mathrm{ad}_\gamma$ is proved in \cite[Propositions 2.2 and 2.5]{cox90}} of $\mathrm{ad}^*_\gamma$, is proved by a problem-specific argument in \cite[Corollary 6.2]{cox90} (see also \cite[Theorem 9.2.3]{henrot06}).

\begin{prp}[{\cite[Corollary 6.2]{cox90}}]\label{crl_den_rel}
Problem \eqref{p_den_rel_min} is a pseudo-concave minimization problem. Moreover, if it admits a minimizer, then it admits a classical $0$-$1$ minimizer that belongs to $\mathrm{ad}_\gamma$.
\end{prp}

When the design domain $\Omega$ is a ball, the explicit, unique, and classical 0-1 solution is known due to symmetry.

\begin{prp}[\cite{krein55} (cf.~{\cite[Theorem 9.4.1]{henrot06}})]\label{prp_ball}
Let $\Omega=\mathcal{B}(0,R)\subset\r^N$ be a Euclidean ball of radius $R>0$ and set $\omega_N\coloneq|\mathcal{B}(0,1)|$. Then
\e{
\overline{\theta}(x)=
\begin{cases}
0 & \mathrm{if}\ |x|\le r,\\
1 & \mathrm{if}\ r<|x|\le R
\end{cases}
\label{analytic_max}
}
with $r=\prn{R^N-\gamma/\omega_N}^{1/N}$ is the unique solution to the maximization problem \eqref{p_den_rel_max}. Moreover,
\e{
\underline{\theta}(x)=
\begin{cases}
1 & \mathrm{if}\ |x|\le r,\\
0 & \mathrm{if}\ r<|x|\le R
\end{cases}
\label{analytic_min}
}
with $r=\prn{\gamma/\omega_N}^{1/N}$ is the unique solution to the minimization problem \eqref{p_den_rel_min}.
\end{prp}

A proof of Proposition \ref{prp_ball} is similar to that of the classical Faber--Krahn inequality and uses the symmetric decreasing rearrangement technique; see \cite{henrot06}.

As shown in numerical experiments in Section \ref{sec_num}, we can actually obtain a classical 0-1 design by a simple gradient method. Since a solution to the unrelaxed problem
\e{
\underset{\chi_\omega\in\mathrm{ad}_\gamma}{\min}\ \underset{u\in H_0^1(\Omega)\backslash\{0\}}{\inf}\frac{\int_{\mathrm{\Omega}}|\nabla u|^2 dx}{\int_{\mathrm{\Omega}}\rho(\chi_\omega)u^2 dx}
\label{p_den}
}
can be computed easily by solving \eqref{p_den_rel_min} (especially, in the case when $\Omega$ is a ball, the explicit solution is known), Problem \eqref{p_den} can be used as a benchmark problem for testing whether heuristic methods can recover a classical $0$-$1$ minimizer\footnote{By heuristic, we mean that (i) they have no convergence guarantee to a stationary point and (ii) existence of classical 0-1 solution is often ignored in the context.}, such as the level-set method \cite{allaire04} and evolutionary algorithms \cite{deaton14}. Also, the maximization problems \eqref{p_relaxed}, \eqref{p_den_rel_max}, and \eqref{p_num_rel_max} have no (suboptimal) local optimal solutions. Thus, they can be used as benchmark problems for testing whether a method reaches a stationary point efficiently, without the additional complication of suboptimal local maximizers.

\subsection{Problem setting in elasticity}

We consider the optimization problem of the first eigenvalue in linear elasticity \cite[Section 4.1.6]{allaire02} (see also \cite{allaire07} for the fundamentals of linear elasticity). We assume that $\Omega\subset\r^N$ is a regular bounded open set and its boundary $\partial\Omega$ is divided into two disjoint parts $\Gamma_D$ and $\Gamma_N$ (Dirichlet and Neumann), where the surface measure of $\Gamma_D$ is nonzero. We define the admissible displacement space by
\e{
V_D\coloneq
\cur{
v\in H^1(\Omega)^N\mid v=0\ \text{on}\ \Gamma_D
},
}
where the boundary condition is understood in the trace sense.

Eigenvalues $\lambda\in\r$ and eigenvectors $u\in V_D$ (typically $N=2$ or $3$) satisfy (in the weak sense)
\e{
\begin{cases}
-\mathrm{div}A^+(\chi_\omega)e(u)=\lambda\rho(\chi_\omega)u & \text{in}\ \Omega\\
A^+(\chi_\omega)e(u)n=0 & \text{on}\ \Gamma_N\\
u=0 & \text{on}\ \Gamma_D
\end{cases}
\label{eq_elas}
}
where $e(u)=(\nabla u + \nabla u^\top)/2$ is the strain tensor and
\e{\ald{
A^+(\chi_\omega(x)) &= A_1+(A_2-A_1)\chi_\omega(x),\\
A_1 &= 2\mu_1I_4+\prn{\kappa_1-\frac{2\mu_1}{N}}I_2\otimes I_2,\\
A_2 &= 2\mu_2I_4+\prn{\kappa_2-\frac{2\mu_2}{N}}I_2\otimes I_2\\
}}
with constants $0<\mu_1\le\mu_2$, $0<\kappa_1\le\kappa_2$, the second-order identity tensor $I_2$, and the identity tensor of the fourth order $I_4$. The density $\rho(\chi_\omega)$ is the same as in the conductivity case \eqref{cond_dens}. The equation \eqref{eq_elas} models the vibration of an $N$-dimensional linear elastic material.

\begin{rmk}
There are two main differences from the conductivity setting in Section \ref{subsec_cond_setting}. One is that the equation \eqref{eq_elas} is a system of equations ($u$ is a vector-valued function, unlike a scalar-valued function in Section \ref{sec_cond}). In this case, the first eigenvalue $\lambda_1(\theta)$ is not necessarily simple (cf.~\cite[Section 7.3.3]{allaire07}), hence $\theta\mapsto\lambda_1(\theta)$ is not necessarily differentiable. The other is that $A_1,A_2$ are not proportional to an identity tensor $I_4$. This leads to a more complicated relaxed problem than the conductivity setting.
\end{rmk}

We consider the following two-phase optimal design problem of linear elastic materials:
\e{
\underset{\chi_\omega\in\mathrm{ad}_\gamma}{\max}\ \underset{u\in V_D\backslash\{0\}}{\inf}\frac{\int_{\mathrm{\Omega}} A^+(\chi_\omega)e(u):e(u) dx}{\int_{\mathrm{\Omega}} \rho(\chi_\omega)|u|^2 dx}.
\label{p_original_elas}
}
In engineering applications, maximization of the first eigenvalue is important since it leads to a structure with high dynamic stiffness.

The relaxed problem of \eqref{p_original_elas} by homogenization is
\e{
\underset{\substack{\theta\in\mathrm{ad}^*_\gamma,\\A^*\in G_\theta}}{\max}\ \underset{u\in V_D\backslash\{0\}}{\inf}\frac{\int_{\mathrm{\Omega}} A^* e(u):e(u) dx}{\int_{\mathrm{\Omega}} \rho(\theta)|u|^2 dx},
}
where $G_\theta$ is the G-closure of two isotropic materials in elasticity \cite[Section 2.3.1]{allaire02}. The relaxed problem can also be simplified by using the Hashin--Shtrikman upper bound $A^*(\theta)e(u):e(u)=\max_{A^*\in G_\theta} A^* e(u):e(u)$ \cite{allaire02}:
\e{
\underset{\theta\in\mathrm{ad}^*_\gamma}{\max}\ \underset{u\in V_D\backslash\{0\}}{\inf}\frac{\int_{\mathrm{\Omega}} A^*(\theta)e(u):e(u) dx}{\int_{\mathrm{\Omega}} \rho(\theta)|u|^2 dx}.
\label{p_relaxed_elas}
}
However, in the elasticity setting, the Hashin--Shtrikman upper bound $A^*(\theta)$ is more complicated than the arithmetic mean $A^+(\theta)$ ($\theta\mapsto A^*(\theta)$ is not affine).

\subsection{Pseudo-concavity results in elasticity}\label{subsec_elas}

Unfortunately, the relaxed problem \eqref{p_relaxed_elas} cannot be expected to be a pseudo-concave optimization since the numerator of the Rayleigh quotient is not affine (nor concave) with respect to $\theta$. However, when we replace $A^*(\theta)$ by the arithmetic mean $A^+(\theta)$, the approximation problem
\e{
\underset{\theta\in\mathrm{ad}^*_\gamma}{\max}\ \underset{u\in V_D\backslash\{0\}}{\inf}\frac{\int_{\mathrm{\Omega}} A^+(\theta)e(u):e(u) dx}{\int_{\mathrm{\Omega}} \rho(\theta)|u|^2 dx},
\label{p_relaxed_elas_pse}
}
is a pseudo-concave optimization problem. Since the arithmetic mean $A^+(\theta)e(u):e(u)$ is always greater than or equal to the Hashin--Shtrikman upper bound $A^*(\theta(x))e(u(x)):e(u(x))$ \cite[Section 2.3.2]{allaire02}, the optimal value of \eqref{p_relaxed_elas_pse} is always greater than or equal to that of \eqref{p_relaxed_elas}. Therefore, the approximation problem \eqref{p_relaxed_elas_pse} gives an upper bound of \eqref{p_relaxed_elas}.

\begin{thm}
Problem \eqref{p_relaxed_elas_pse} is a pseudo-concave maximization problem, and thus every stationary point is a global maximizer. Moreover, the optimal value of Problem \eqref{p_relaxed_elas_pse} is no less than the optimal value of Problem \eqref{p_relaxed_elas}.
\end{thm}
\begin{proof}
Problem \eqref{p_relaxed_elas_pse} satisfies Assumption \ref{asm1} with $H=L^2(\Omega)^N$, $V=V_D$, $\ang{A(\theta)u,u}_H=\int_\Omega A^+(\theta)e(u):e(u)dx$, and $\ang{B(\theta)u,u}_H=\int_\Omega \rho(\theta)|u|^2dx$. The ellipticity of the operator $A(\theta)$ follows from Korn's inequality (cf.~\cite[Lemma 5.3.3]{allaire07}). Operators in the numerator and the denominator are both affine with respect to $\theta$. Thus, the first part is a direct consequence of Theorem \ref{thm_main} and Proposition \ref{prp_psemax}. The second part follows from the arithmetic mean bound \cite[Section 2.3.2]{allaire02}.
\end{proof}

In a very special case when $A_1=A_2$ (when $\theta$ only appears in the denominator of the Rayleigh quotient), the relaxed problem becomes
\e{
\underset{\theta\in\mathrm{ad}^*_\gamma}{\max}\ \underset{u\in V_D\backslash\{0\}}{\inf}\frac{\int_{\mathrm{\Omega}} A_1 e(u):e(u) dx}{\int_{\mathrm{\Omega}} \rho(\theta)|u|^2 dx},
\label{p_relaxed_elas_den}
}
which is a pseudo-concave maximization problem.

\begin{crl}
Problem \eqref{p_relaxed_elas_den} is a pseudo-concave maximization problem, and thus every stationary point is a global maximizer.
\end{crl}

The existence of suboptimal local maxima in the relaxed problem \eqref{p_relaxed_elas} is still open \cite{allaire02}. It is not even known whether a simpler compliance minimization problem has a suboptimal local minimum or not in the elasticity setting. The difficulty comes from the nonlinearity of the homogenized tensor $A^*(\theta)$ w.r.t.~$\theta$.

\section{Numerical experiments}\label{sec_num}

We provide simple numerical results on relaxed optimal design problems in the conductivity setting: the pseudo-concave maximization problem \eqref{p_relaxed} and the pseudo-concave minimization problem \eqref{p_den_rel_min} in a circular domain.

We do not treat relaxed optimal design problems in elasticity since nonsmooth optimization algorithms in the non-reflexive Banach space $L^\infty(\Omega)$ have not been studied theoretically to the best of the author's knowledge, and thus they are out of the scope of this research (see also Remark \ref{rmk_reflex}).

\subsection{Implementation}

In this section, we explain the implementation of an optimization algorithm. See also the FreeFEM code in Appendix \ref{apn}.

For an optimization algorithm, we use the projected gradient method \cite{blank17}
\e{
\theta_{k+1}=\Pi_{\mathrm{ad}^*_\gamma}(\theta_k \pm \alpha_k\nabla \lambda_1(\theta_k)),
\label{pg}
}
where the subscript $k$ is the iteration counter, $\Pi_{\mathrm{ad}^*_\gamma}$ is the projection operator onto the set $\mathrm{ad}^*_\gamma$, $\alpha_k>0$ is a stepsize, and $\pm$ corresponds to maximization and minimization of $\lambda_1(\theta)$, respectively.

Since the first eigenvalue is simple for the conductivity with a connected domain considered below, the objective function $\theta\mapsto\lambda_1(\theta)$ of Problem \eqref{p_relaxed} is continuously differentiable. By Proposition \ref{prp_cla}, its gradient is
\e{
\nabla\lambda_1(\theta)=(c_2-c_1)|\nabla u|^2-\lambda_1(\theta)(\rho_2-\rho_1)u^2,
\label{grad_lambda}
}
where $u$ is normalized by $\int_\Omega\rho(\theta)u^2dx=1$. Set $c_1=c_2$ for Problem \eqref{p_den_rel_min}.

\begin{rmk}\label{rmk_reflex}
Since the optimization variable $\theta$ belongs to a non-reflexive Banach space $L^\infty(\Omega)$, the projected gradient method \eqref{pg} needs to be modified for a theoretical study. The gradient of the objective function does not necessarily belong to the space of the optimization variable. Indeed, the gradient defined by \eqref{grad_lambda} belongs to $L^1(\Omega)\subsetneq (L^\infty(\Omega))^*$ but not necessarily belongs to $L^\infty(\Omega)$. For a theoretical study of the projected gradient method in a non-reflexive Banach space, see \cite{blank17}. In our numerical experiments, the addition and subtraction of the optimization variable and the gradient, which belong to different finite element spaces \texttt{Xh} and \texttt{Vh}, are calculated automatically in FreeFEM (see the FreeFEM code in Appendix \ref{apn}). The validity of this discretization is out of the scope of this paper.
\end{rmk}

We set the parameters as follows. The volume parameter is $\gamma=0.5|\Omega|$, stepsize for the projected gradient method is $0.1$, $\rho_1=0.3$ and $\rho_2=0.7$, $c_1=0.5$ and $c_2=1$ for \eqref{p_relaxed} and $c_1=c_2=1$ for \eqref{p_den_rel_min}. We use the uniform initial design $\theta(x)=\gamma/|\Omega|$. We discretize $\Omega$ into a triangular mesh and use $P_0$ (piecewise constant) finite elements for $\theta$ and $P_2$ (piecewise quadratic) finite elements for $u$. We use $P_2$ instead of $P_1$ to avoid the numerical instability called checkerboard pattern \cite{bendsoe04}.

All the experiments were conducted on a MacBook Pro (2019, 1.4 GHz Quad-Core Intel Core i5, 8 GB memory) with FreeFEM v.4.6 \cite{freefem,hecht12}. 

\subsection{Results}

Numerical solutions to Problems \eqref{p_relaxed} and \eqref{p_den_rel_min} in a circular domain $\Omega$ obtained by $200$ iterations of the projected gradient method are shown in Figure \ref{fig}.

Theorem \ref{thm_psemax} suggests that Figure \ref{fig}(a) is a global maximizer. As Proposition \ref{crl_den_rel} suggests, Figure \ref{fig}(b) shows almost a classical $0$-$1$ solution. Also, Figure \ref{fig}(b) is consistent with Krein's result (Proposition \ref{prp_ball}).

\begin{figure}[H]
  \begin{tabular}{cc}
  \begin{minipage}[t]{0.45\hsize}
    \centering
    \includegraphics[width=6cm]{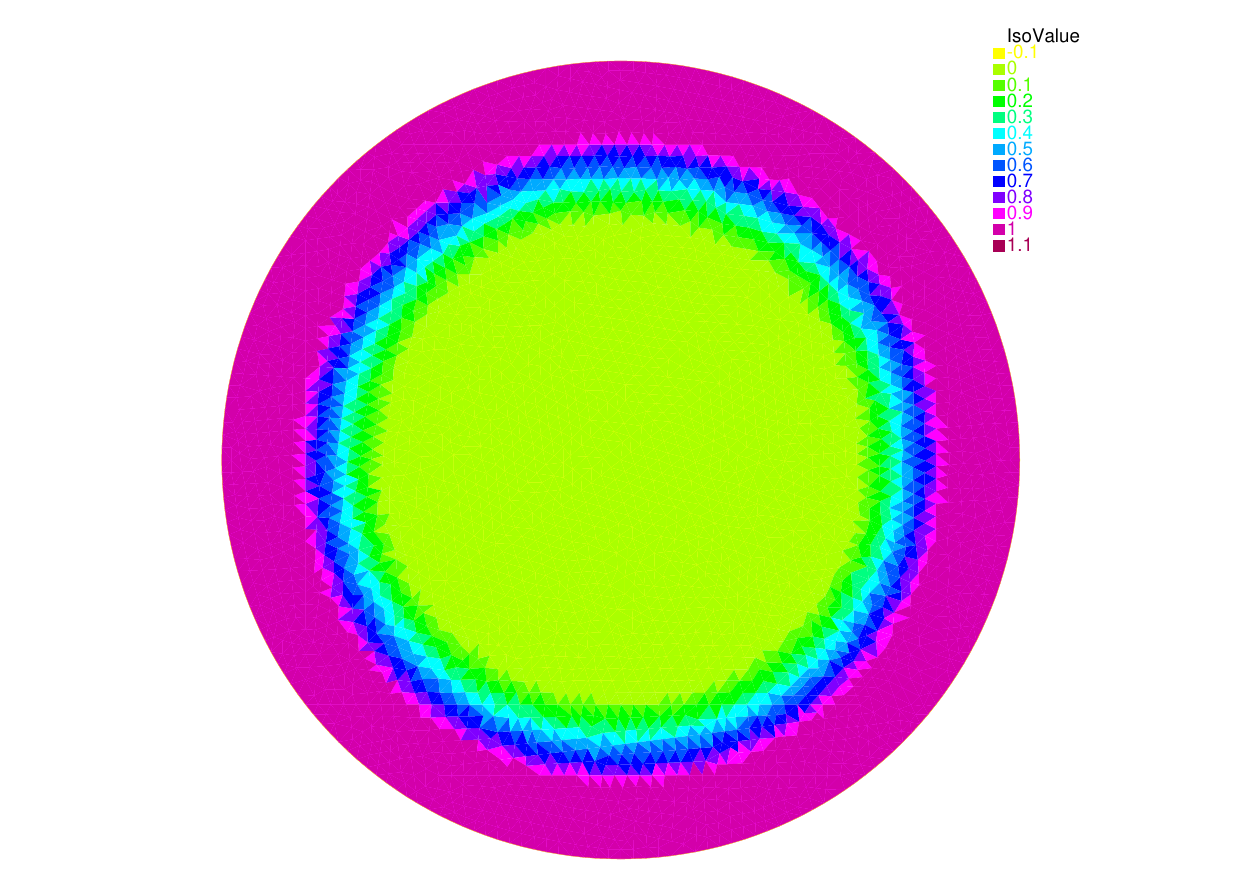}
    \subcaption{pseudo-concave max.}
  \end{minipage}
  &
  \hspace{-3mm}
  \begin{minipage}[t]{0.45\hsize}
    \centering
    \includegraphics[width=6cm]{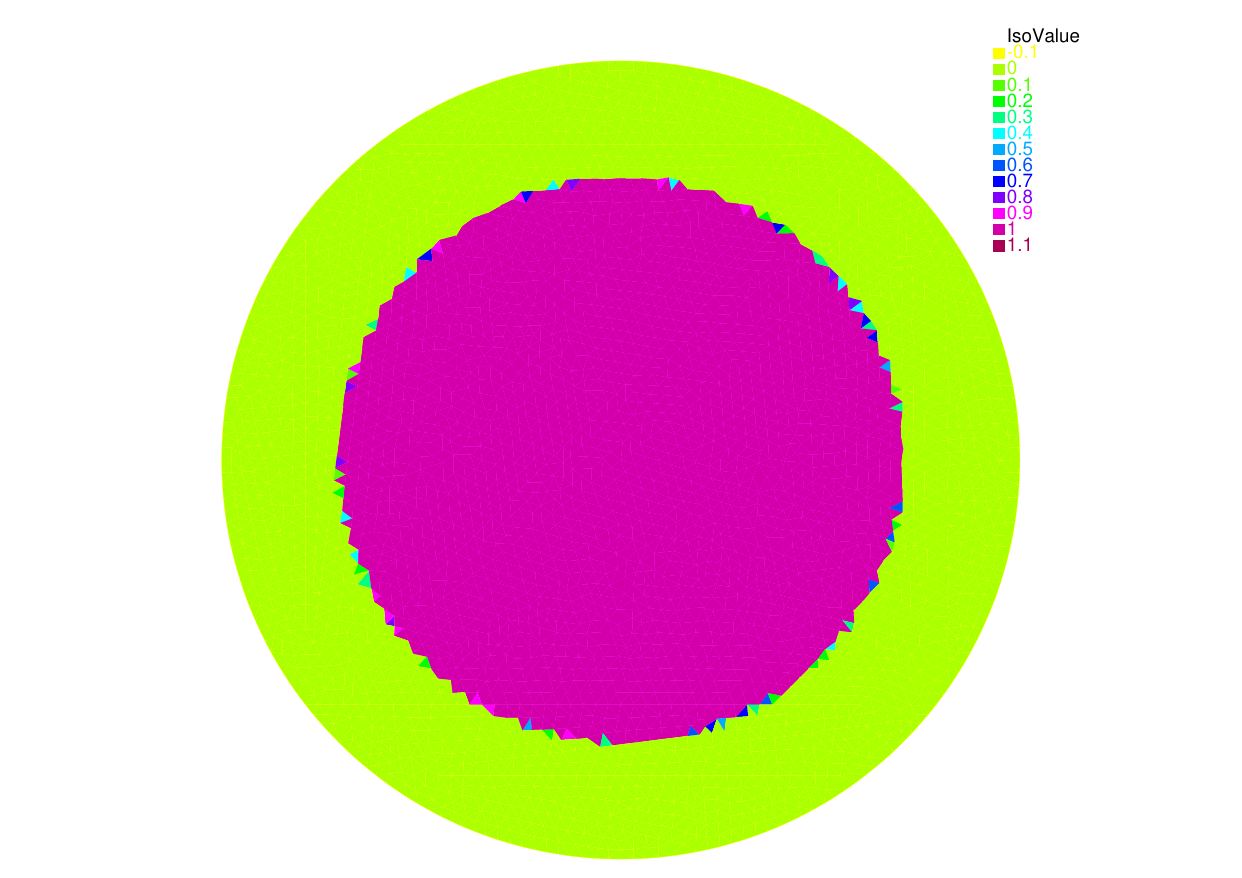}
    \subcaption{pseudo-concave min.}
  \end{minipage}
  \end{tabular}
  \caption{(a) A numerical solution to the pseudo-concave maximization problem \eqref{p_relaxed}, which is expected to be a global maximizer. (b) A numerical solution to the pseudo-concave minimization problem \eqref{p_den_rel_min}, which is almost a classical $0$-$1$ design. It is consistent with Proposition \ref{prp_ball}.}
  \label{fig}
\end{figure}

\section{Conclusion}

We showed that the first eigenvalue of a certain linear elliptic operator is a pseudo-concave function with respect to a density-like parameter. This is a generalization of a classical result by \cite{jouron78} to the case where the first eigenvalue is not necessarily simple (hence not necessarily differentiable). We also applied the result to optimal design problems (topology optimization problems) relaxed by the homogenization method to show that certain problems are pseudo-concave optimization problems. Due to pseudo-concavity, in maximization problems, every stationary point is a global maximizer. We numerically illustrate that the computed solutions are consistent with the theoretical global maximizers and classical $0$-$1$ minimizers in the conductivity setting. These problems can be used as benchmark problems to test heuristic topology optimization methods used in engineering.

More thorough studies on optimal design problems of elasticity are needed, due to both theoretical and algorithmic challenges. On the theoretical side, the potential non-existence of suboptimal (i.e., non-global) local optima in optimal design problems of elasticity may require the consideration of a broader class of functions than pseudo-concave functions such as invex or incave functions \cite{mishra08}, in which every stationary point is guaranteed to be globally optimal. On the algorithmic side, further research is needed into nonsmooth optimization methods in non-reflexive Banach spaces.

\begin{appendices}

\section{FreeFEM code}\label{apn}

All the numerical results are produced using the following code in FreeFEM without any additional packages. Set \texttt{c1}$=$\texttt{c2}$=1$ for Problem \eqref{p_den_rel_min} and change the sign of \texttt{stepsize*gradient} from \texttt{+} to \texttt{-} for minimization.\\

\begin{lstlisting}
verbosity = 0;
// Parameters
real stepsize = 0.1;
real frac = 0.5; // Volume fraction of material 2
real c1 = 0.5;
real c2 = 1;
real rho1 = 0.3;
real rho2 = 0.7;
int maxiter = 200;
real[int] levels = [-0.1,0,0.1,0.2,0.3,0.4,0.5,0.6,0.7,0.8,0.9,1,1.1];

// Mesh & FE spaces
border Gamma(t=0, 2*pi) {x=cos(t); y=sin(t);};
mesh Th = buildmesh(Gamma(200));
fespace Vh(Th, P2), Xh(Th, P0);
Vh u, v;
Xh ones=1, theta0, gradient; 
Xh theta=frac; // Uniform initial design

// Other settings
macro grad(u) [dx(u),dy(u)] //
real vol0 = int2d(Th)(1), mult, mult0, mult1, err, normu;
real[int] eval(1); // to store eigenvalues
Vh[int] evec(1); // to store eigenvectors

// Optimization
for (int i = 0; i < maxiter; ++i){
    // Objective value computation
    varf op (u, v)
        = int2d(Th)(((c2-c1)*theta+c1)*(grad(u)'*grad(v)))
        + on(1, u=0);
    varf b ([u], [v]) = int2d(Th)(((rho2-rho1)*theta+rho1)*u*v);
    matrix OP = op(Vh, Vh, solver=Crout, factorize=1);
    matrix B = b(Vh, Vh, solver=CG, eps=1e-20);
    int k = EigenValue(OP, B, sym=true, value=eval, vector=evec,
        tol=1e-10, maxit=0, ncv=0);
    // Gradient computation & Update
    u = evec[0];
    normu = sqrt(int2d(Th)(((rho2-rho1)*theta+rho1)*u*u));
    u = u/normu;
    gradient = (c2-c1)*grad(u)'*grad(u)-eval[0]*(rho2-rho1)*u*u;
    theta = theta + stepsize*gradient; // -/+ = minimize/maximize
    // Projection by the bisection method
    err = 1;
    real[int] t = theta[];
    mult0 = t.min-1;
    mult1 = t.max;
    theta0 = theta;
    while (abs(err) > 1e-7){
        mult = (mult0 + mult1)/2;
        theta = theta0 - mult;
        theta = min(theta,1);
        theta = max(theta,0);
        err = int2d(Th)(theta)/vol0-frac;
        if (err > 0){
            mult0 = mult;
        }
        else{
            mult1 = mult;
        }
    }
    // Display & Plot
    cout << "Iter. " << i << ", lambda1 = " << eval[0] << endl;
    plot(theta, cmm="Iter."+i, wait=false, viso=levels, fill=true, value=true);
}

// Plot & Save
plot(theta, wait=true, viso=levels, fill=true, value=true, ps="density.eps");
// // Plot eigenfunction
// plot(evec[0],wait=true, fill=true, value=true, dim=3);
\end{lstlisting}

\section*{Acknowledgement}

The author would like to thank Professor Yoshihiro Kanno for his feedback on this work. The work of the author is partially supported by JSPS KAKENHI JP25KJ0120. 

\section*{Declarations}

The author has no conflict of interest.

\section*{Data availability statement}

All the numerical results can be reproduced using the code in Appendix \ref{apn}.

\end{appendices}

\bibliography{ref}%

\end{document}